\documentclass[12pt,a4paper]{article}
\usepackage{amsmath,amssymb,amsthm}
\usepackage{color}
\usepackage{soul}

\newtheorem{theorem}{Theorem}[section]
\newtheorem{hyp}[theorem]{Hypothesis}
\newtheorem{lemma}[theorem]{Lemma}
\newtheorem{corollary}[theorem]{Corollary}
\newtheorem{proposition}[theorem]{Proposition}
\theoremstyle{definition}

\numberwithin{equation}{section}

\usepackage{titlesec}
\titleformat{\section}{\bfseries}{\thesection}{1em}{}
\titleformat{\subsection}{\itshape}{\thesubsection}{1em}{}

\newcommand{\R}{\mathbb{R}}
\newcommand{\real}{\mathbb{R}}
\newcommand{\N}{\mathbb{N}}
\newcommand{\dd}{\hskip0.2mm\mbox{\rm d}}
\def\expe{\mathrm{e}}

\def\moT#1{\left|#1\right|_{[0,T]}}
\def\mot#1{\left|#1\right|_{[0,t]}}
\def\MoT#1{\left\|#1\right\|_{[0,T]}}
\def\Mot#1{\left\|#1\right\|_{[0,t]}}
\def\for{\mbox{ for }\ }

\def\dehat{\hat}
\def\tht{\mathbf{u}}
\def\vtht{\mathbf{u}}

\def\be{\begin{equation}\label}
\def\ee{\end{equation}}

\def\play{\mathfrak{p}}

\begin{document}

\title{Inverse parameter-dependent Preisach operator in thermo-piezoelectricity modeling
\thanks{This research was supported by RVO: 67985840.}}
\author{Pavel~Krej\v{c}\'{i} and Giselle A. Monteiro
\thanks{Institute of Mathematics, Czech Academy of Sciences,
\v{Z}itn\'{a} 25, 115 67 Praha 1, Czech Republic, e-mail:~{\tt krejci@math.cas.cz},
{\tt gam@math.cas.cz}.}}
\date{}

\maketitle

\begin{abstract}
Hysteresis is an important issue in modeling piezoelectric materials, for example,
in applications to energy harvesting, where hysteresis losses may influence the efficiency of the process.
The main problem in numerical simulations is the inversion of the underlying hysteresis operator.
Moreover, hysteresis dissipation is accompanied with heat production, which in turn increases the
temperature of the device and may change its physical characteristics. More accurate models therefore
have to take the temperature dependence into account for a correct energy balance.
We prove here that the classical Preisach operator with a fairly general parameter-dependence
admits a Lipschitz continuous inverse in the space of right-continuous regulated functions,
propose a time-discrete and memory-discrete inversion algorithm, and show that higher regularity of the inputs leads
to a higher regularity of the output of the inverse.
\end{abstract}
\noindent
{\it Key words:} Hysteresis, Preisach operator, inversion formula, piezoelectricity.

\noindent
{\it MSC(2010):} {34C55, 47J40, 58C07, 49J40}

%%%%%%%%%%%%%%%%%%%%%%%%%%%%%%%%%%%%%%%%%%%%%%%%%%%%

\section*{Introduction}

The Preisach model for ferromagnetic hysteresis was introduced in \cite{pr} and is still very popular
in both the engineering and the theoretical community for its analytical simplicity and good agreement
with experiments. For inverse control and hysteresis compensation, it is of interest to have some information
about the possibility of theoretical and numerical inversion. An explicit inversion formula is known
for a special case of the Preisach model, the so-called Prandtl-Ishlinskii operator, see \cite{mathz}.
The first result about the existence and analytical properties of a general inverse Preisach operator was published
in \cite{bv}.

The original definition of \cite{pr} and the analysis in \cite{bv} are based on the concept of a {\it two-parameter relay\/}.
Here, we use the formalism of \cite{max}, where a variational definition was shown to be equivalent.
More specifically, the Preisach model is represented in terms of the system
of {\it play operators\/} as solution operators of rate-independent variational inequalities, see Section \ref{vari} below.
In this setting, the Preisach model can be interpreted as a nonlinear generalization of the Prandtl-Ishlinskii operator,
which can be viewed as a linear functional applied to the system of play operators. In special cases of superpositions
of Prandtl-Ishlinskii operators with monotone functions, which also belong to the class of Preisach operators,
an explicit representation of the inverse is still available, see \cite{ku,vs}. In general, no closed form
formula for the inverse Preisach operator is known and numerical approximation schemes have to be used, see, e. g., \cite{error}.

The inversion problem is even more involved if the Preisach operator depends on additional parameter functions.
The situation is favorable in the ``almost Prandtl-Ishlinskii'' case, where an explicit inversion formula
is still available, see \cite{adk}. In the full Preisach case, the existence of a Lipschitz continuous inverse
was proved in \cite{KK} when the Preisach density is just multiplied by a parameter function like in Eq.~\eqref{pe4} below.
Here, we extend the result to the case of general parameter dependence.

The paper is organized as follows. In Section \ref{phys}, we give a brief physical motivation for the study
of inverse parameter-dependent Preisach operators. Section \ref{vari} is devoted to a survey of results
on scalar evolution variational inequalities in the Kurzweil-Stieltjes integral setting in the space of
regulated functions, that is, functions of one variable which admit at each point of their domain of definition
both one-sided limits. The parameter-dependent Preisach operator is introduced in Section \ref{prei}.
Section \ref{inve} represents the core of the paper and brings the main results on the existence and Lipschitz continuity
of the inverse parameter-dependent Preisach operator in the space of regulated functions. In the final Section \ref{regu}
it is shown that if the input and the parameter functions are continuous (resp. absolutely continuous),
then the output of the inverse parameter-dependent Preisach operator is continuous (resp. absolutely continuous)
and the corresponding estimates remain valid.

%%%%%%%%%%%%%%%%%%%%%%%%%%%%%%%%%%%%%%%%%%%%%%%%%%%%

\section{Physical motivation}\label{phys}

Multifunctional materials transform spontaneously mechanical energy into the electromagnetic one and vice versa.
They find numerous engineering applications in high precision devices like actuators, sensors, or energy harvesters.
This is related to phenomena called piezoelectricity or magnetostriction, and an interested reader can find
more information in \cite{dgv,kamlah,ku,dkprv}. Most such materials exhibit a particular kind of hysteresis:
ferromagnetic hysteresis in the electromagnetic component and what is called ``butterfly hysteresis''
in the magnetoelastic component. It has been observed that the hysteresis loops at different pre-stresses
exhibit a kind of self-similarity. This leads to the hypothesis that all hysteresis effects are due to one single
scalar hysteresis operator acting on one auxiliary self-similar variable,
and a model for magnetostrictive hysteresis based on this hypothesis was proposed in \cite{dkv}
together with parameter identification for data obtained from measurements of Galfenol,
which is one of the well-known magnetostrictive alloys. Let $\mathcal{P}$ be a hysteresis operator, and let
$m, h, \varepsilon, \sigma$ denote the magnetization, magnetic field, strain, and stress, respectively, in 1D
experimental setting. Considering $h$ and $\sigma$ as state variables, the constitutive law was proposed in \cite{dkv}
in the form
\be{pe1}
\begin{array}{rl}
h &= \mathcal{P}[q],\\[2mm]
\varepsilon &= -f'(\sigma) \mathcal{V}[q],\\[2mm]
q &= \frac{h}{f(\sigma)},
\end{array}
\ee
where $\mathcal{V}$ is the hysteresis potential associated with the operator $\mathcal{P}$, see \cite{max}, and $f$
is a self-similarity function which is to be identified from experiments. The model was adapted in \cite{KK} to the case of piezoelectricity
for the functions $P, E$ representing the polarization and electric field, respectively, with $E$ and $\varepsilon$ as
state variables, and the constitutive equations have the form
\be{pe2}
\begin{array}{rl}
P &= \mathcal{P}[q],\\[2mm]
\varepsilon &= f'(\varepsilon) \mathcal{V}[q],\\[2mm]
q &= \frac{E}{f(\varepsilon)}.
\end{array}
\ee
In both cases, $\mathcal{P}$ was chosen to be the Preisach operator introduced in \cite{pr}, and the main reason was,
apart from its nice analytical properties, that an explicit formula for the potential $\mathcal{V}$ is available, see \cite{max}.

More accurate experiments have shown certain discrepancies between the model and the experiments at low fields.
In particular, the above model is not able to explain the phenomenon of mechanical demagnetization or depolarization.
This is why a state dependent mean field feedback term $\alpha(\varepsilon)$ was introduced in \cite{KK18} as a correction
to \eqref{pe2} in the following form:
\be{pe3}
\begin{array}{rl}
P &= \mathcal{P}[q],\\
\varepsilon &= f'(\varepsilon) \mathcal{V}[q] + \frac12\alpha'(\varepsilon)\mathcal{P}^2[q],\\
q &= \frac{1}{f(\varepsilon)}(E - \alpha(\varepsilon)P).
\end{array}
\ee
We see that in this case, the self-similar variable $q$ can be determined in terms of the state variables
as a solution of the equation 
\be{pe4}
q + \frac{\alpha(\varepsilon)}{f(\varepsilon)} \mathcal{P}[q] = \frac{E}{f(\varepsilon)}.
\ee
The solvability of such equations has been established in \cite{KK}. However, since hysteresis energy dissipation produces heat
and the material parameters depend on the temperature, more accurate models have to take the temperature dependence into
account, for example by considering a temperature-dependent Preisach operator $\mathcal{P}(\theta)[q]$ as in \cite{km}, where $\theta$
is the absolute temperature. Then, Eq.~\eqref{pe4} becomes
\be{pe5}
q + \frac{\alpha(\varepsilon)}{f(\varepsilon)} \mathcal{P}(\theta)[q] = \frac{E}{f(\varepsilon)},
\ee
for which the argument of \cite{KK} does not work and has to be refined.

%%%%%%%%%%%%%%%%%%%%%%%%%%%%%%%%%%%%%%%%%%%%%%%%%%%%

\section{A Kurzweil-Stieltjes integral variational inequality}\label{vari}

In this section we survey known results (see, e.~g., \cite{kuhys}) on the so-called play operator $\play_r$
with threshold $r>0$ in the space of regulated functions defined on a given interval $[0,T]$. Recall that
a function $w: [0,T] \to \real$ is said to be regulated, if both one-sided limits $w(t+), w(t-)$ exist
at every point $t\in [0,T]$ with the convention $w(T+) = w(T)$, $w(0-) = w(0)$. We denote by $G[0,T]$ the
space of regulated functions and by $G_R[0,T]$ the space of right-continuous regulated functions on $[0,T]$.
Indeed, the space $BV[0,T]$ of functions of bounded variation on $[0,T]$ is contained in $G[0,T]$.

With a given $q\in G_R[0,T]$ and a parameter $r>0$, the mapping $\play_r$ associates the solution 
$\xi_r\in BV_R[0,T]:= BV[0,T]\cap G_R[0,T]$ of the integral variational inequality
\begin{equation}\label{play}
\left\{\begin{array}{ll}
|q(t)-\xi_r(t)|\leq r, &  t\in [0,T],
\\[2mm]
\int_0^T\big(q(t)-\xi_r(t)-z(t)\big)\,\dd \xi_r(t)\ge 0,& z\in G[0,T],\, 
\moT{z}\le r, 
\\[2mm]
\xi_r(0)=\max\{q(0)-r\,,\,\min\{0\,,\,q(0)+r\}\,\},
\end{array}
\right.
\end{equation}
where the integral is understood as the Kurzweil-Stieltjes integral. 
According to a more general theory developed in \cite{kl1},
a unique solution $\xi_r=\play_r[q]\in BV_R[0,T]$ of \eqref{play} exists and satisfies the inequality
\begin{equation}\label{v-ineq}
\int_s^t\big(q(\tau)-\xi_r(\tau)-z(\tau)\big)\,\dd \xi_r(\tau)\ge 0,\quad  \forall z\in G[0,T], \moT{z}\le 1,
\end{equation}
for every $0\le s<t\le T$.

The play $\play_r:G_R[0,T]\to BV_R[0,T]$ is Lipschitz continuous
with respect to the sup-norm. More specifically, we endow the space $G[0,T]$ with a family of seminorms
$|w|_{[s,t]} := \sup_{\tau \in [s,t]} |w(\tau)|$. Then for $q_1, q_2 \in G_R[0,T]$ and $t\in[0,T]$ we have
\begin{equation}\label{play-Lip}
|\play_r[q_1](t)-\play_r[q_2](t)|\leq \mot{q_1-q_2}.
\end{equation}
For $0\le t<t+\delta\le T$ and $q \in G_R[0,T]$ we may put in \eqref{play-Lip} $q_1 = q$ and
$$
q_2(s) = 
\left\{
\begin{array}{ll}
q(s) & \for \ s \in [0,t]\,,\\
q(t) & \for \ s \in (t, T]\,. 
\end{array}
\right.
$$
Then by \eqref{play-Lip} the following inequality holds
\begin{equation}\label{play-l}
|\play_r[q](t)-\play_r[q](t+\delta)|\leq |q(\cdot)-q(t)|_{[t,t+\delta]}.
\end{equation}
As a consequence of \eqref{play-l}, we see that $\play_r[q]$ is continuous if $q$ is continuous, and
$\xi_r = \play_r[q]$ is absolutely continuous if $q$ is absolutely continuous. In the latter case,
it follows from \eqref{play-l} that $|\dot\xi_r(t)|\le |\dot q(t)|$ a.\,e., and
\begin{equation}\label{varic}
\dot\xi_r(t)(q(t) - \xi_r(t) - z) \ge 0 \ \mbox{ a.\,e. }\ \forall z \in [-r,r].
\end{equation}
If $q:[0,T]\to\R$ is a right-continuous step function associated with a division $0=s_0<s_1<\dots<s_N=T$ of $[0,T]$
such that
\begin{equation}\label{q-step}
q(t)=\sum_{n=1}^N q_{n-1}\chi_{[s_{n-1},s_n)}(t)+q_{N}\chi_{[s_N]}(t),\quad t\in[0,T],
\end{equation}
then the solution $\xi_r=\play_r[q]$ of the variational inequality \eqref{play} is given by the explicit formula
\begin{equation}\label{play-step}
\xi_r(t)=\sum_{n=1}^N\xi^r_{n-1}\chi_{[s_{n-1},s_n)}(t)+\xi^r_{N}\chi_{[s_N]}(t),\quad t\in[0,T],
\end{equation}
where $\xi^r_{0}=\max\{q_0-r\,,\,\min\{0,q_0+r\}\,\}$ and
\begin{equation}\label{xi-n}
\xi^r_{n}=\max\{q_n-r\,,\,\min\{\xi^r_{n-1},q_n+r\}\,\},\quad n=1,\dots,N.
\end{equation}
Formula \eqref{xi-n} can be equivalently stated in variational form
\begin{equation}\label{v-ineq-n}
(\xi^{r}_n-\xi^{r}_{n-1})\,(q_n-\xi^{r}_n-r\,z)\ge 0\quad\mbox{for \ }z\in\R,\,\,|z|\le 1,\quad n=1,\dots,N.
\end{equation}
The following result is proved in \cite[Lemma 4.4]{KLR-arXiv} as a special case of the so-called Brokate's identity.

\begin{lemma}\label{Brokate-id}
Let $\rho>0$ and $\sigma>0$ be given, and let $q:[0,T]\to\R$ be a right-continuous step function as in \eqref{q-step}. 
Put $\xi=\play_\rho[q]$ and $\eta=\play_\sigma[\xi]$. Then $\eta=\play_{\rho+\sigma}[q]$ and the implication
\begin{equation}\label{B-id}
\eta_{n}-\eta_{n-1}\ne 0\quad\Rightarrow\quad (\xi_{n}-\xi_{n-1})\,(\eta_{n}-\eta_{n-1})>0 
\end{equation}
holds for every $n=1,\dots,N$.
\end{lemma}

We reformulate the above statement for arbitrary right-continuous regulated inputs.

\begin{lemma}\label{lem3.5}
Let $0<r_1<r_2$ and $q\in G_R[0,T]$ be given, and let $\xi_{r_i}\in BV_R[0,T]$
be the respective solutions of \eqref{play} with $r=r_i$, $i=1,2$. Then
\begin{equation}\label{eq34}
\left\{\begin{array}{ll}
|\xi_{r_2}(t)-\xi_{r_1}(t)|\leq r_2-r_1, &  t\in [0,T],
\\[2mm]
\int_0^T\big(\xi_{r_1}(t)-\xi_{r_2}(t)-z(t)\big)\,\dd \xi_{r_2}(t)\ge 0,& \forall z\in G[0,T],\, 
\moT{z}\le r_2-r_1.
\end{array}
\right.
\end{equation}
\end{lemma}

\begin{proof}
Since $q\in G_R[0,T]$, there exists a sequence $\{q_n\}$ of right-continuous step functions such that 
$\moT{q_n-q} \to 0$. Denote $\xi^1_n=\play_{r_1}[q_n]$ and $\xi^2_n=\play_{r_2}[q_n]$. By Lemma \ref{Brokate-id},
we have $\xi^2_n=\play_{r_2-r_1}[\xi^1_n]$. The Lipschitz continuity \eqref{play-Lip} of the play operator ensures that
$\moT{\xi^1_n-\xi_{r_1}}\to 0$. We can therefore pass to the limit as $n \to \infty$ and conclude that
$\xi_{r_2}=\play_{r_2-r_1}[\xi_{r_1}]$, which is \eqref{eq34} by definition. 
\end{proof}

%%%%%%%%%%%%%%%%%%%%%%%%%%%%%%%%%%%%

\section{Preisach operator}\label{prei}

Let $\psi:\R\times\R_+\times\R\to\R$ be a measurable function. 
For an input $q\in G_R[0,T]$ and parameter vector function $\tht\in G_R([0,T];\real^L)$ we define the output of a
parameter-dependent Preisach operator $\mathcal{P}$ by the integral formula
\begin{equation}\label{P}
\mathcal{P}(\tht)[q](t)=\int_0^\infty\int_0^{\xi_r(t)} \psi(\tht(t),r,v)\,\dd v\,\dd r,
\end{equation}
where $\xi_r=\play_r[q]$ is the solution of the integral variational inequality \eqref{play}. Clearly, for constant parameter functions $\tht$ we retrieve the definition of the classical Preisach operator as presented in \cite{max}.

To ensure that the definition \eqref{P} is meaningful we adopt the following hypothesis.

\begin{hyp}\label{basic} We assume that for all $\vtht \in \real^L$ we have
\begin{enumerate}
\item[$(i)$] $0\leq \psi(\vtht,r,v)\le \mu(r)$ a.e., where $\mu\in L^1(0,\infty)$ and $M=\int_0^\infty\mu(r)\dd r$;
\item[$(ii)$] $\left\|\nabla_\vtht\psi(\vtht,r,v)\right\|\le K(r,v)$ a.e., where 
$K\in L^1((0,\infty)\times\R)$, $\|\cdot\|$ denotes the Euclidean norm in $\real^L$, and 
\[
M_1=\int_0^\infty\int_{-\infty}^\infty K(r,v)\dd v\,\dd r.
\]
\end{enumerate}
\end{hyp}

It follows immediately from the definition \eqref{play} that for each $q\in G_R[0,T]$,
$r\geq\moT{q}$, and $t\in[0,T]$ we have $\xi_r(t)=\play_r[q](t)= 0$.
Hence in \eqref{P} we integrate over bounded intervals, and Hypothesis \ref{basic}\,(i)
guarantees that $\mathcal{P}(\tht)[q](t)$ is well defined. For $(\vtht,r,s)\in \R^L\times\R_+\times\R$ denote
\begin{equation}\label{g}
g(\vtht,r,v)=\int_0^v\psi(\vtht,r,s)\,\dd s.
\end{equation}
Since $\psi\ge 0$ a.\,e., the mapping $v\in\R\mapsto g(\vtht,r,v)$ is nondecreasing for each $\vtht\in\R^L$ and $r>0$.
Moreover, Hypotheses \ref{basic}\,(i) and (ii) imply that 
\begin{equation}\label{g-ineq}
|g(\vtht_1,r,v_1)-g(\vtht_2,r,v_2)|\leq \mu(r)\,|v_1-v_2|+K^*(r)\,\|\vtht_1-\vtht_2\|,
\end{equation}
for every $\vtht_1,\vtht_2\in \real^L$, $v_1,v_2\in\R$, and $r> 0$, with $K^*(r)=\int_{-\infty}^\infty K(r,s)\,\dd s$.

Formula \eqref{P} can be rewritten in the form
\begin{equation}\label{Pg}
\mathcal{P}(\tht)[q](t)=\int_0^\infty g(\tht(t),r,\play_r[q](t))\,\dd r
\end{equation}
for $\tht \in G_R([0,T];\real^L)$, $q\in G_R[0,T]$, and $t \in [0,T]$.
From \eqref{g-ineq} it also follows that the operator $\mathcal{P}$ is Lipschitz continuous in the sense that 
\begin{equation}\label{P-Lip}
\big|\mathcal{P}(\tht_1)[q_1](t)-\mathcal{P}(\tht_2)[q_2](t)\big|
	\leq M\,\mot{q_1-q_2}+ M_1\,\|\tht_1(t)-\tht_2(t)\|
\end{equation}
for every $q_1,q_2\in G_R[0,T]$, $\tht_1,\tht_2\in G_R([0,T];\real^L)$, and $t\in[0,T]$.
Similarly as in \eqref{play-l} we also have 
\begin{equation}\label{P-reg}
\big|\mathcal{P}(\tht)[q](t)-\mathcal{P}(\tht)[q](s)\big|
	\le M\,|q(\cdot)-q(s)|_{[s,t]} + M_1\,\|\tht(t)-\tht(s)\|,
\end{equation}
for every $0\leq s<t\leq T$. Therefore, the integral expression \eqref{P} defines an operator
$\mathcal{P}: G_R([0,T];\real^L)\times G_R[0,T]\to G_R[0,T]$ which is Lipschitz continuous in its domain of definition.
In addition, if both $\tht$ and $q$ are continuous functions (respect. absolutely continuous), then so is $P(\tht)[q]$.

%%%%%%%%%%%%%%%%%%%%%%%%%%%%%%%%%%%%%%%%%

\section{Existence of the inverse}\label{inve}

Given $\tht\in G_R([0,T];\real^L)$, $w\in G_R[0,T]$, we will show that there exists a unique $q\in G_R[0,T]$ such that
\begin{equation}\label{eq}
w(t)=q(t)+\mathcal{P}(\tht)[q](t),\quad t\in[0,T].
\end{equation}

\medskip

Firstly, assume that both $w$ and $\tht$ are right-continuous step functions and  
let $0=s_0<s_1<\dots<s_N=T$ be a division such that
\[
w(t)=\sum_{n=1}^Nw_{n-1}\chi_{[s_{n-1},s_n)}(t)+w_{N}\chi_{[s_N]}(t),
\quad
\tht(t)=\sum_{n=1}^N\tht_{n-1}\chi_{[s_{n-1},s_n)}(t)+\tht_{N}\chi_{[s_N]}(t).
\]
We will prove by induction that, in this case, a function $q$ satisfying \eqref{eq} is of the form 
\[
q(t)=\sum_{n=1}^Nq_{n-1}\chi_{[s_{n-1},s_n)}(t)+q_{N}\chi_{[s_N]}(t),\quad t\in[0,T].
\]
Recall that, for such a $q$, the function $\xi_r=\play_r[q]$, $r\ge 0$, is given by 
\[
\xi_r(t)=\sum_{n=1}^N\xi^r_{n-1}\chi_{[s_{n-1},s_n)}(t)+\xi^r_{N}\chi_{[s_N]}(t), \quad t\in[0,T],
\]
with $\xi^r_{n}\in\R$ defined by \eqref{xi-n}. Thus, rewriting \eqref{eq} in $[0,s_1)$ we obtain
\[
w_0=q_0+\int_0^\infty g(\tht_0,r,\xi_0^r)\,\dd r,
\]
where $g$ is the function given by \eqref{g}. 
The existence of a unique $q_0$ satisfying the equality above is guaranteed by the fact that the right-hand side defines an increasing continuous function of $q_0$.

Assume that we have determined $q_n\in\R$, $n=0,1,\dots, k-1$, with $k>0$, such that the step function $q(t)$ satisfies \eqref{eq} on $[0,s_{k})$. If $k<N$, put
\[
V_{k+1}=\{\dehat{q}\in G[0,s_k+1]\,:\,\dehat{q}(t)=q(t)+\dehat{q}_k\chi_{[s_k,s_{k+1})}(t)\,\},
\]
otherwise, for $k=N$ consider
\[
V_{N}=\{\dehat{q}\in G[0,s_N]\,:\,\dehat{q}(t)=q(t)+\dehat{q}_N\chi_{[s_N]}(t)\,\}.
\]
Thus, a function $\dehat{q}$ in $V_{k+1}$ (respect. in $V_N$) satisfies \eqref{eq} in $[0,s_{k+1})$ (respect. in $[0,s_N]$) if the following equality holds
\[
w_k=\dehat{q}_k+\int_0^\infty g(\tht_k,r,\dehat{\xi}_k^r)\,\dd r,
\]
with $\dehat{\xi}^r_{k}=\max\{\dehat{q}_k-r\,,\,\min\{\xi^r_{k-1},\dehat{q}_k+r\}\}$.
The monotonicity of $g(\tht_k,r,\cdot)$ and of the `dead zone function'
$\dehat{q} \mapsto \max\{\dehat{q}-r\,,\,\min\{\xi^r_{k-1},\dehat{q}+r\}\}$
ensure that the right-hand side defines an increasing continuous mapping of $\dehat{q}_k$.
Thus, there exists a unique $q_k$ satisfying the equality above, and consequently we obtain a function $q$
satisfying \eqref{eq} in either $[0,s_{k+1})$ or $[0,s_N]$. The induction argument over $k$
completes the proof of the following result:

\begin{theorem}\label{exist-step}
Assume that $\psi$ satisfies Hypothesis \ref{basic}. If $\tht:[0,T]\to\R^L$, $w:[0,T]\to\R$ are right-continuous step functions, then there exists a unique right-continuous step function $q:[0,T]\to\R$ satisfying \eqref{eq}.
\end{theorem}

Next we consider discrete Preisach operators of the form 
\begin{equation}\label{disc}
\mathcal{P}_k(\tht)[q](t)=\sum_{j=1}^k g_j(\tht(t),\xi_{r_j}(t)),\quad % q\in G_R[0,T],\,
t\in[0,T],\,k\in\N,
\end{equation}
where $0<r_1<r_2<\dots$, and $g_j:\R^{L+1}\to\R$, $j\in\N$, are functions fulfilling the following conditions:

\begin{hyp}\label{H2}
Assume that $g_j(\vtht,0)=0$ for all $\vtht\in\R^L$, and
\begin{enumerate}
\item[$(i)$] $g_j$ is nondecreasing in the second variable;
\item[$(ii)$] there exist positive numbers $\mu_j$ and $K_j$ such that 
\begin{equation}\label{lip}
|g_j(\vtht_1,v_1)-g_j(\vtht_1,v_2)|\leq K_j\,\|\vtht_1-\vtht_2\|+\mu_j\,|v_1-v_2|,
\end{equation}
for every $\vtht_1,\vtht_2\in \real^L$, $v_1,v_2\in\R$.
\end{enumerate}
\end{hyp}

Clearly, the argument of the proof of Theorem \ref{exist-step} remains valid for Eq.~\eqref{eq} as well
if the operator $\mathcal{P}$ is replaced with $\mathcal{P}_k$. In other words, for each given right-continuous
step functions $\tht:[0,T]\to\R^L$, $w:[0,T]\to\R$ there exists a unique right-continuous step function 
$q:[0,T]\to\R$ such that %for all $t \in [0,T]$ we have
\be{diseq}
q(t)+\mathcal{P}_k(\tht)[q](t)=w(t),\quad t\in[0,T].
\ee
To establish the continuity of the inverse of the operator $I+\mathcal{P}_k$,	 the following lemma is crucial.

\begin{lemma}\label{p36a}
Let Hypotheses \ref{H2} (i), (ii) hold, let $\mathcal{P}_k$ be the mapping defined by \eqref{disc},
and  let $\tht:[0,T]\to\R^L$, $w,\,\dehat{w}:[0,T]\to\R$ be right-continuous step functions. 
If $q,\,\dehat{q}\in G_R[0,T]$ satisfy
\begin{equation}\label{eq41}
\begin{split}
q(t)+\mathcal{P}_k(\tht)[q](t)=w(t),
\\[2mm]
\dehat{q}(t)+\mathcal{P}_k(\tht)[\dehat{q}](t)=\dehat{w}(t),
\end{split}
\quad\qquad t\in[0,T],
\end{equation}
for some $k\in\N$, then for each $t\in[0,T]$ we have
\begin{equation}\label{eq42}
|q(t)-\dehat{q}(t)|\le \rho_k\,\mot{w-\dehat{w}},
	\quad\mbox{where \ }\rho_k=\prod_{j=1}^k(1+\mu_j).
\end{equation}
\end{lemma}

\begin{proof}
Let $0=s_0<s_1<\dots<s_N=T$ be a division of $[0,T]$ such that
\[
\tht(t)=\sum_{n=1}^N\tht_{n-1}\chi_{[s_{n-1},s_n)}(t)+\tht_{N}\chi_{[s_N]}(t),
\]
\[
w(t)=\sum_{n=1}^Nw_{n-1}\chi_{[s_{n-1},s_n)}(t)+w_{N}\chi_{[s_N]}(t),
\quad
\dehat{w}(t)=\sum_{n=1}^N\dehat{w}_{n-1}\chi_{[s_{n-1},s_n)}(t)+\dehat{w}_{N}\chi_{[s_N]}(t).
\]
According to Theorem \ref{exist-step}, we have
\[
q(t)=\sum_{n=1}^Nq_{n-1}\chi_{[s_{n-1},s_n)}(t)+q_{N}\chi_{[s_N]}(t),
\quad
\dehat{q}(t)=\sum_{n=1}^N\dehat{q}_{n-1}\chi_{[s_{n-1},s_n)}(t)+\dehat{q}_{N}\chi_{[s_N]}(t),
\]
and consequently the corresponding $\xi_r=\play_r[q]$ and $\dehat{\xi}_r=\play_r[\dehat{q}]$, for $r\ge 0$, read as follows
\[
\xi_r(t)=\sum_{n=1}^N\xi^r_{n-1}\chi_{[s_{n-1},s_n)}(t)+\xi^r_{N}\chi_{[s_N]}(t),
\quad
\dehat{\xi}_r(t)=\sum_{n=1}^N\dehat{\xi}^r_{n-1}\chi_{[s_{n-1},s_n)}(t)+\dehat{\xi}^r_{N}\chi_{[s_N]}(t).
\]
Note that for $t=0$, we have
\[
w_0-\dehat{w}_0=q_0-\dehat{q}_0+\sum_{j=1}^k(g_j(\tht_0,\xi^{r_j}_0)-g_j(\tht_0,\dehat{\xi}^{r_j}_0)).
\]
Thus, the monotonicity of $g_j(\tht_0,\cdot)$ and of the initial value mapping imply that
\begin{equation}\label{q_0}
|q_0-\dehat{q}_0|\leq |w_0-\dehat{w}_0|,
\end{equation}
and \eqref{eq42} holds under the standard convention $\prod_{j=1}^0(1+\mu_j) = 1$.

Assume $t\in(0,T)$ and let 
$\ell\in\{0,1,\dots,N-1\}$ be such that $t\in[s_{\ell},s_{\ell+1})$ 
(the case $t=T$ can be treated in similar way). 
We will prove \eqref{eq42} by induction over $k$.  
For $k=1$ the equalities in \eqref{eq41} are  equivalent to 
\begin{equation*}\label{eq43}
q_n+g_1(\tht_n,\xi^{r_1}_n)=w_n,
\quad\dehat{q}_n+g_1(\tht_n,\dehat{\xi}^{r_1}_n)=\dehat{w}_n,\quad\mbox{for \ }n=0,1,\dots,N.
\end{equation*}
This together with \eqref{lip} implies
\begin{align*}
|q(t)-\dehat{q}(t)|=|q_\ell-\dehat{q}_\ell|
	&\le |w_\ell-\dehat{w}_\ell|+|g_1(\tht_\ell,\xi^{r_1}_\ell)-g_1(\tht_\ell,\dehat{\xi}^{r_1}_\ell)|
	\\&\le 
\mot{w-\dehat{w}}+\mu_1\,|\xi^{r_1}_\ell-\dehat{\xi}^{r_1}_{\ell}|.
\end{align*}
It remains to show that $|\xi^{r_1}_\ell-\dehat{\xi}^{r_1}_{\ell}|\leq \mot{w-\dehat{w}}$. Note that, by applying \eqref{v-ineq-n} with an appropriate choice of $z\in\R$ we obtain
\[
(\xi^{r_1}_n-\xi^{r_1}_{n-1})\,\big(q_n-\dehat{q}_n-(\xi^{r_1}_n-\dehat{\xi}^{r_1}_{n})\big)\ge 0,
\quad
(\dehat{\xi}^{r_1}_{n}-\dehat{\xi}^{r_1}_{n-1})\,\big(\dehat{q}_n-q_n-(\dehat{\xi}^{r_1}_{n}-\xi^{r_1}_n)\big)\ge 0,
\]
therefore, for every $n=1,\dots,N$, 
\begin{equation}\label{eq46}
\big((\xi^{r_1}_n-\dehat{\xi}^{r_1}_{n})-(\xi^{r_1}_{n-1}-\dehat{\xi}^{r_1}_{n-1})\big)
	\,\big(q_n-\dehat{q}_n-(\xi^{r_1}_n-\dehat{\xi}^{r_1}_{n})\big)\ge 0.
%\quad n=1,\dots,N.
\end{equation}
Let $v_n=\max\{\,|\xi^{r_1}_n-\dehat{\xi}^{r_1}_{n}|\,, \mot{w-\dehat{w}}\}$, $n=0,1,\dots,\ell$.
We claim that $v_n\le v_{n-1}$ for  every $n=1,\dots,\ell$. Otherwise, if $v_n>v_{n-1}$
for some $n=1,\dots,\ell$, we would have
\begin{equation}\label{eq49}
|\xi^{r_1}_n-\dehat{\xi}^{r_1}_{n}|>\mot{w-\dehat{w}}
	\quad\mbox{and}\quad 
|\xi^{r_1}_n-\dehat{\xi}^{r_1}_{n}|>|\xi^{r_1}_{n-1}-\dehat{\xi}^{r_1}_{n-1}|.
\end{equation}
It is not hard to see that the second inequality yields
\[
\big((\xi^{r_1}_n-\dehat{\xi}^{r_1}_{n})-(\xi^{r_1}_{n-1}-\dehat{\xi}^{r_1}_{n-1})\big)
\,(\xi^{r_1}_n-\dehat{\xi}^{r_1}_{n}) > 0.
\]
This together with \eqref{eq46} implies 
\[
(\xi^{r_1}_n-\dehat{\xi}^{r_1}_{n})\,\big(q_n-\dehat{q}_n-(\xi^{r_1}_n-\dehat{\xi}^{r_1}_{n})\big)\ge 0,
\]
that is,
\begin{equation}\label{eq51}
(\xi^{r_1}_n-\dehat{\xi}^{r_1}_{n})\,
\big(w_n-\dehat{w}_n-(g_1(\tht_n,\xi^{r_1}_n)-g_1(\tht_n,\dehat{\xi}^{r_1}_n))-(\xi^{r_1}_n-\dehat{\xi}^{r_1}_{n})\big)\ge 0.
\end{equation}
Since $g_1$ is nondecreasing in the second variable we have
\[
(g_1(\tht_n,\xi^{r_1}_n)-g_1(\tht_n,\dehat{\xi}^{r_1}_n))(\xi^{r_1}_n-\dehat{\xi}^{r_1}_{n})\ge 0,
\]
and \eqref{eq51} reduces to
\[
(\xi^{r_1}_n-\dehat{\xi}^{r_1}_{n})^2\leq(\xi^{r_1}_n-\dehat{\xi}^{r_1}_{n})\,(w_n-\dehat{w}_n).
\]
Therefore $|\xi^{r_1}_n-\dehat{\xi}^{r_1}_{n}|\le\mot{w-\dehat{w}}$, contradicting \eqref{eq49}. Hence $\{v_n:n=0,1,\dots,\ell\}$ is nondecreasing and, in particular
\[
|\xi^{r_1}_\ell-\dehat{\xi}^{r_1}_{\ell}|
	\leq \max\{\,|\xi^{r_1}_0-\dehat{\xi}^{r_1}_{0}|\,,\,\mot{w-\dehat{w}}\,\}.
\]
Since by \eqref{play-Lip} and \eqref{q_0} we have
\[
|\xi^{r_1}_0-\dehat{\xi}^{r_1}_{0}|\leq |q_0-\dehat{q}_0|
	\leq |w_0-\dehat{w}_0|\leq \mot{w-\dehat{w}},
\]
it follows that $|\xi^{r_1}_\ell-\dehat{\xi}^{r_1}_{\ell}|\le \mot{w-\dehat{w}}$,
and we conclude that \eqref{eq42} holds for $k=1$.

Let $k\in\N$, $k>1$, and assume that \eqref{eq42} holds for $k-1$. Rewriting the equations in \eqref{eq41} we obtain
\begin{equation*}\label{eq55a}
\begin{split}
q(t)+\mathcal{P}_{k-1}(\tht)[q](t)=w(t)-g_k(\tht(t),\xi_{r_k}(t)),
\\[2mm]
\dehat{q}(t)+\mathcal{P}_{k-1}(\tht)[\dehat{q}](t)=\dehat{w}(t)-g_k(\tht(t),\dehat{\xi}_{r_k}(t)),
\end{split}
%\quad t\in[0,T],
\end{equation*}
which, by the induction hypothesis and \eqref{lip}, yields
\begin{equation}\label{eq56}
|q(t)-\dehat{q}(t)|\le \rho_{k-1}\,
	\big(\mot{w-\dehat{w}}+\mu_k\,\big|\xi_{r_k}-\dehat{\xi}_{r_k}\big|_{[0,t]}\,\big).
\end{equation}
Thus, it suffices to show that $\big|\xi_{r_k}-\dehat{\xi}_{r_k}\big|_{[0,t]}\leq \mot{w-\dehat{w}}$ or, equivalently
\begin{equation}\label{eq56b}
|\xi_n^{r_k}-\dehat{\xi}_n^{r_k}|\leq \mot{w-\dehat{w}}\quad\mbox{for every \ }n=0,1,\dots,\ell.
\end{equation}
Applying \eqref{v-ineq-n} with an appropriate choice of $z\in\R$, for each $n=1,\dots,N$ we get
\[
(\xi^{r_k}_n-\xi^{r_k}_{n-1})\,\big(q_n-\dehat{q}_n-(\xi^{r_k}_n-\dehat{\xi}^{r_k}_{n})\big)\ge 0,
\qquad
(\dehat{\xi}^{r_k}_{n}-\dehat{\xi}^{r_k}_{n-1})\,\big(\dehat{q}_n-q_n-(\dehat{\xi}^{r_k}_{n}-\xi^{r_k}_n)\big)\ge 0,
\]
which implies
\begin{equation}\label{eq61}
\big((\xi^{r_k}_n-\dehat{\xi}^{r_k}_{n})-(\xi^{r_k}_{n-1}-\dehat{\xi}^{r_k}_{n-1})\big)
	\,\big(q_n-\dehat{q}_n-(\xi^{r_k}_n-\dehat{\xi}^{r_k}_{n})\big)\ge 0,
\quad n=1,\dots, N.
\end{equation}
On the other hand, for each $j=1,\dots, k-1,$ by applying Lemma \ref{Brokate-id} with $\rho=r_j$ and $\sigma=r_k-r_j$ we get $\xi_{r_k}=\play_{r_k-r_j}[\xi_{r_j}]$, which yields
\[
(\xi^{r_k}_n-\xi^{r_k}_{n-1})\,(\xi^{r_j}_n-\xi^{r_k}_n-(r_k-r_j)\,z)\ge 0
	\quad\mbox{for \ }z\in\R,\,\,|z|\le 1,\quad n=1,\dots,N.
\]
Thus, the choice $z=\frac{\dehat{\xi}^{r_j}_n-\dehat{\xi}^{r_k}_n}{r_k-r_j}$ leads to
\[
(\xi^{r_k}_n-\xi^{r_k}_{n-1})\,\big(\xi^{r_j}_n-\dehat{\xi}^{r_j}_{n}-(\xi^{r_k}_n-\dehat{\xi}^{r_k}_{n})\big)\ge 0,\quad j=1,\dots,k-1.
\]
Similarly, we deduce that
\[
(\dehat{\xi}^{r_k}_n-\dehat{\xi}^{r_k}_{n-1})\,\big(\dehat{\xi}^{r_j}_{n}-\xi^{r_j}_n-(\dehat{\xi}^{r_k}_{n}-\xi^{r_k}_n)\big)\ge 0,
\]
and consequently
\begin{equation}\label{eq62}
\big((\xi^{r_k}_n-\dehat{\xi}^{r_k}_{n})-(\xi^{r_k}_{n-1}-\dehat{\xi}^{r_k}_{n-1})\big)
	\,\big(\xi^{r_j}_n-\dehat{\xi}^{r_j}_{n}-(\xi^{r_k}_n-\dehat{\xi}^{r_k}_{n})\big)\ge 0,
\end{equation}
for every $n=1,\dots,N,$ and $j=1,\dots, k-1$. 
Let $v_n=\max\{\,|\xi^{r_k}_n-\dehat{\xi}^{r_k}_{n}|\,, \mot{w-\dehat{w}}\}$, $n=0,1,\dots,\ell$.
We claim that $v_n\le v_{n-1}$ for every $n=1,\dots,\ell$. Otherwise, if $v_n>v_{n-1}$ for some $n=1,\dots,N$,
we would have
\[
|\xi^{r_k}_n-\dehat{\xi}^{r_k}_{n}|>\mot{w-\dehat{w}}
	\quad\mbox{and}\quad 
|\xi^{r_k}_n-\dehat{\xi}^{r_k}_{n}|>|\xi^{r_k}_{n-1}-\dehat{\xi}^{r_k}_{n-1}|.
\]
It is not hard to see that the second inequality yields
\begin{equation}\label{eq70}
\big((\xi^{r_k}_n-\dehat{\xi}^{r_k}_{n})-(\xi^{r_k}_{n-1}-\dehat{\xi}^{r_k}_{n-1})\big)
	\,(\xi^{r_k}_n-\dehat{\xi}^{r_k}_{n})\ge 0.
\end{equation}
This together with \eqref{eq61} implies 
\[
(\xi^{r_k}_n-\dehat{\xi}^{r_k}_{n})\,\big(q_n-\dehat{q}_n-(\xi^{r_k}_n-\dehat{\xi}^{r_k}_{n})\big)\ge 0,
\]
that is,
\begin{equation}\label{eq71}
(\xi^{r_k}_n-\dehat{\xi}^{r_k}_{n})\,\left(w_n-\dehat{w}_n
	-\sum_{j=1}^k\big(g_j(\tht_n,\xi^{r_j}_n)-g_j(\tht_n,\dehat{\xi}^{r_j}_n)\big)
		-(\xi^{r_k}_n-\dehat{\xi}^{r_k}_{n})\right)\ge 0.
\end{equation}
Note that \eqref{eq62} and \eqref{eq70} imply that 
\[
(\xi^{r_k}_n-\dehat{\xi}^{r_k}_{n})
	\,\big(\xi^{r_j}_n-\dehat{\xi}^{r_j}_{n}-(\xi^{r_k}_n-\dehat{\xi}^{r_k}_{n})\big)\ge 0,
	\quad j=1,\dots, k-1,
\]
that is, $(\xi^{r_k}_n-\dehat{\xi}^{r_k}_{n})(\xi^{r_j}_n-\dehat{\xi}^{r_j}_{n})\ge 0$. Using this and the fact that  each $g_j$ is nondecreasing in the second variable we obtain
\[
(g_j(\tht_n,\xi^{r_j}_n)-g_j(\tht_n,\dehat{\xi}^{r_j}_n))(\xi^{r_k}_n-\dehat{\xi}^{r_k}_{n})\ge 0,\quad j=1,\dots, k,
\]
Therefore \eqref{eq71} reduces to
\[
(\xi^{r_k}_n-\dehat{\xi}^{r_k}_{n})^2\leq(\xi^{r_k}_n-\dehat{\xi}^{r_k}_{n})\,(w_n-\dehat{w}_n),
\]
which implies $|\xi^{r_k}_n-\dehat{\xi}^{r_k}_{n}|\le\mot{w-\dehat{w}}$, a contradiction.
Hence $\{v_n:n=0,1,\dots,\ell\}$ is nondecreasing and, in particular,
\[
|\xi^{r_k}_n-\dehat{\xi}^{r_k}_{n}|\leq \max\{\,|\xi^{r_k}_0-\dehat{\xi}^{r_k}_{0}|\,, \mot{w-\dehat{w}}\},
\quad n=1,\dots,\ell.
\]
From \eqref{play-Lip} and \eqref{q_0} it follows that $|\xi^{r_k}_0-\dehat{\xi}^{r_k}_{0}|\leq |w_0-\dehat{w}_0|$. Thus \eqref{eq56b} holds, and \eqref{eq42} follows from \eqref{eq56}. 
\end{proof}

\begin{corollary}\label{c1}
Let Hypotheses \ref{H2} (i), (ii) hold, let $\mathcal{P}_k$ be the mapping defined by \eqref{disc},
and  let $\tht, \dehat{\tht}:[0,T]\to\R^L$, $w, \dehat{w}:[0,T]\to\R$ be right-continuous step functions. 
Let $q, \dehat{q}\in G_R[0,T]$ be solutions of the equations
\begin{equation}\label{ce1}
\begin{split}
q(t)+\mathcal{P}_k(\tht)[q](t)=w(t),
\\[2mm]
\dehat{q}(t)+\mathcal{P}_k(\dehat{\tht})[\dehat{q}](t)=\dehat{w}(t),
\end{split}
\quad\qquad t\in[0,T],
\end{equation}
for some $k\in\N$, then for each $t\in[0,T]$ we have
\begin{equation}\label{ce2}
|q(t)-\dehat{q}(t)|\le \rho_k\,\left(\mot{w-\dehat{w}} 
	+ \left(\sum_{j=1}^k K_j\right)\Mot{\tht - \dehat{\tht}}\right)
\end{equation}
with $\rho_k$ as in Lemma \ref{p36a}.
\end{corollary}

\begin{proof}
Given an arbitrary $t\in[0,T]$, we rewrite \eqref{ce1} in the form
\begin{equation}\label{ce3}
\begin{split}
&q(t)+\mathcal{P}_k(\tht)[q](t)=w(t),
\\[2mm]
&\dehat{q}(t)+\mathcal{P}_k(\tht)[\dehat{q}](t)=\dehat{w}(t) + \mathcal{P}_k(\tht)[\dehat{q}](t)
- \mathcal{P}_k(\dehat{\tht})[\dehat{q}](t),
\end{split}
\end{equation}
hence, by Lemma \ref{p36a},
\be{ce4}
|q(t)-\dehat{q}(t)|\leq\rho_k\,\big(\mot{w-\dehat{w}}+\mot{(\mathcal{P}_k(\tht)-\mathcal{P}_k(\dehat{\tht}))[\dehat{q}]}\big).
\ee
{}From \eqref{disc}--\eqref{lip} it follows that
\begin{align*}
\mot{\mathcal{P}_k(\tht)[q]-\mathcal{P}_k(\dehat{\tht})[\dehat{q}]}
&=\mot{\sum_{j=1}^k (g_j(\tht,\play_{r_j}[\dehat{q}])-g_j(\dehat{\tht},\play_{r_j}[\dehat{q}]))}
\\ &\le \Big(\sum_{j=1}^k K_j\Big)\,\Mot{\tht- \dehat{\tht}},
\end{align*}
which we wanted to prove.
\end{proof}

\begin{corollary}\label{c2}
Let Hypotheses \ref{H2} (i), (ii) hold, let $\mathcal{P}_k$ be the mapping defined by \eqref{disc}, and 
let $\tht,\dehat{\tht}\in G_R([0,T];\R^L)$, \ $w, \dehat{w}\in G_R[0,T]$ be given. 
Then there exist solutions $q, \dehat{q}\in G_R[0,T]$ of the equations \eqref{ce1} and the inequality \eqref{ce2}
holds for all $t\in[0,T]$.
\end{corollary}

\begin{proof}
We find sequences $\tht_i, \dehat{\tht}_i:[0,T]\to\R^L$, $w_i, \dehat{w}_i:[0,T]\to\R$
of right-continuous step functions, $i \in \N$, such that 
%$\moT{w - w_i} \to 0$, $\moT{\dehat{w} - \dehat{w}_i} \to 0$, $\MoT{\tht - \tht_i} \to 0$, 
%$\MoT{\dehat{\tht} - \dehat{\tht}_i} \to 0$ 
\[
\moT{w - w_i} \to 0, \quad \moT{\dehat{w} - \dehat{w}_i} \to 0, \quad 
\MoT{\tht - \tht_i} \to 0, \quad \MoT{\dehat{\tht} - \dehat{\tht}_i} \to 0
\]
as $i \to \infty$. Inequality
\eqref{ce2} implies that for every $t \in [0,T]$ and every $i, m \in \N$ we have
\begin{align*}
|q_i(t)-q_m(t)|
&\le \rho_k\,\left(\moT{w_i-w_m} + \left(\sum_{j=1}^k K_j\right)\MoT{\tht_i - \tht_m}\right),\\
|\dehat{q}_i(t)-\dehat{q}_m(t)|
&\le \rho_k\,\left(\moT{\dehat{w}_i-\dehat{w}_m}
+ \left(\sum_{j=1}^k K_j\right)\MoT{\dehat{\tht}_i - \dehat{\tht}_m}\right).
\end{align*}
In particular, $\{q_i\}$, $\{\dehat{q}_i\}$ are Cauchy sequences in $G_R[0,T]$ and their limits $q$, $\dehat{q}$
satisfy the equations \eqref{ce1}. To complete the proof, it suffices to pass to the limit as $i \to \infty$
in the inequality
\be{ce4a}
|q_i(t)- \dehat{q}_i(t)|\leq\rho_k\,\Big(\mot{w_i-\dehat{w}_i} + \Big(\sum_{j=1}^k K_j\Big)\,\Mot{\tht_i- \dehat{\tht}_i}\Big)
\ee
which is a consequence of \eqref{ce2}.
\end{proof}

We now pass from the discrete Preisach operator \eqref{disc} to the continuous operator \eqref{P}.

\begin{lemma}\label{p34}
Assume that $\psi$ satisfies Hypothesis \ref{basic}, and let $\mathcal{P}$ be the operator \eqref{P}.
Then for each $k \in \N$ there exists a discrete operator $\mathcal{P}_k$ of the form \eqref{disc} such that
for every $R>0$, $\tht\in G_R([0,T];\R^L)$, and $q\in G_R[0,T]$ such that $\moT{q} \le R$ we have
\begin{equation}\label{approx}
\moT{\mathcal{P}(\tht)[q]-\mathcal{P}_k(\tht)[q]} \le \frac{MR}{k}.
\end{equation}
\end{lemma}

\begin{proof}
Consider a division $0=r_0<r_1<\dots<r_k=R$ with $r_j = \frac{jR}{k}$. Let $P_k$ be
discrete Preisach operator \eqref{disc} with the choice 
\[
g_j(\tht,v)=\int_{r_{j-1}}^{r_{j}}\int_0^v\psi(\tht,r,s)\,\dd s\,\dd r,\quad j=1,\dots,k.
\]
Hypotheses \ref{basic} (i) and (ii) guarantee that the functions $g_j$ satisfy Hypothesis \ref{H2}, with
\be{muj}
K_j=\int_{r_{j-1}}^{r_{j}}\int_{-\infty}^{\infty}K(r,v)\dd v\,\dd r, \ \ \mu_j=\int_{r_{j-1}}^{r_{j}}\mu(r)\dd r.
\ee 
Denote $g(\tht,r,v)=\int_0^v\psi(\tht,r,s)\,\dd s$. We have $\play_r[q](t) = 0$ for all $r \ge R$ and all $t \in [0,T]$,
hence, using \eqref{g-ineq} we obtain that
\begin{align*}
\big|\mathcal{P}(\tht)[q]-\mathcal{P}_k(\tht)[q]\big|(t)
&=\Bigg|\int_0^\infty g(\tht(t),r,\play_r[q](t))\,\dd r-\sum_{j=1}^k g_j(\tht(t),\play_{r_j}[q](t))\Bigg|
\\&
=\Big|\sum_{j=1}^k\int_{r_{j-1}}^{r_{j}}\big(g(\tht(t),r,\play_{r}[q](t))-g(\tht(t),r,\play_{r_j}[q](t))\big)\,\dd r\Big|
\\&
\le \sum_{j=1}^k\int_{r_{j-1}}^{r_{j}}\mu(r)\,\moT{\play_{r}[q]-\play_{r_j}[q]}\,\dd r.
\end{align*}
Note that we have, by Lemma \ref{lem3.5}, that
$\moT{\play_{r}[q]-\play_{r_j}[q]}\le r_j-r\le r_j-r_{j-1} = \frac{R}{k}$ for $r\in[r_{j-1},r_j]$. 
Hence
\[
\moT{\mathcal{P}(\tht)[q]-\mathcal{P}_k(\tht)[q]}
\leq \sum_{j=1}^k(r_j-r_{j-1})\,\int_{r_{j-1}}^{r_{j}} \mu(r)\,\dd r
< \frac{R}{k}\int_0^R \mu(r)\,\dd r \le \frac{MR}{k},
\]
concluding the proof.
\end{proof}

We now state and prove the main result of this paper.

\begin{theorem}\label{p37}
Assume that $\psi$ satisfies Hypothesis \ref{basic}, and let $\tht,\,\dehat{\tht},\,w,\,\dehat{w}\in G_R[0,T]$ be given.
Then there exist solutions $q,\,\dehat{q}\in G_R[0,T]$ of the equations
\begin{equation}\label{eq77}
\begin{split}
q(t)+\mathcal{P}(\tht)[q](t)=w(t),
\\[2mm]
\dehat{q}(t)+\mathcal{P}(\dehat{\tht})[\dehat{q}](t)=\dehat{w}(t),
\end{split}
\qquad t\in[0,T],
\end{equation}
and the inequality
\begin{equation}\label{eq78}
|q(t)-\dehat{q}(t)|\le \expe^M\,\big(\mot{w-\dehat{w}} + M_1\,\Mot{\tht-\dehat{\tht}}\big)
\end{equation}
holds for each $t\in[0,T]$.
\end{theorem}

\begin{proof}
We define the operators $\mathcal{P}_k$ as in Lemma \ref{p34}, put $R = \expe^M \max\{\moT{w}, \moT{\dehat{w}}\}$,
and define $q_k$, $\dehat{q}_k$ as solutions to the equations
\begin{equation}\label{eq77a}
\begin{split}
q_k(t)+\mathcal{P}_k(\tht)[q_k](t)=w(t),
\\[2mm]
\dehat{q}_k(t)+\mathcal{P}_k(\dehat{\tht})[\dehat{q}_k](t)=\dehat{w}(t).
\end{split}
\end{equation}
For $k,m \in \N$ we can write
$$
q_m(t)+\mathcal{P}_k(\tht)[q_m](t)=w(t) + \mathcal{P}_k(\tht)[q_m](t)-\mathcal{P}_m(\tht)[q_m](t),
$$
{}From Corollary \ref{c1} it follows that %for all $t\in [0,T]$ that
\be{ci1}
\moT{q_k-q_m}\le \rho_k\,\moT{\mathcal{P}_k(\tht)[q_m]-\mathcal{P}_m(\tht)[q_m]}.
\ee
Moreover, using Corollary \ref{c1} again for $\dehat{q} = \dehat{w} = 0$ and $\dehat{\tht} = \tht$, we obtain
$$
\moT{q_m} \le \rho_m \moT{w}.
$$
By \eqref{eq42} and \eqref{muj}, we have $\log\rho_m = \sum_{j=1}^m \log(1+\mu_j) \le \sum_{j=1}^m \mu_j \le \int_0^\infty \mu(r)\dd r = M$,
hence $\moT{q_m} \le \expe^M \moT{w} \le R$. We similarly have $\moT{\dehat{q}_m} \le R$. Lemma \ref{p34}
and \eqref{ci1} thus yield
\be{ci2}
\moT{q_k-q_m}\le MR\Big(\frac{1}{k} + \frac{1}{m}\Big)\expe^{M}, %\moT{w},
\ee
and similarly
\be{ci3}
\moT{\dehat{q}_k-\dehat{q}_m}\le MR\Big(\frac{1}{k} + \frac{1}{m}\Big)\expe^{M}. %\moT{w}.
\ee
We see that $\{q_k\}$, $\{\dehat{q}_k\}$ are Cauchy sequences in $G_R[0,T]$.
Let us rewrite \eqref{eq77a} as follows
\begin{equation}\label{eq77b}
\begin{split}
q_k(t)+\mathcal{P}(\tht)[q_k](t)=w(t) + \mathcal{P}(\tht)[q_k](t) - \mathcal{P}_k(\tht)[q_k](t),
\\[2mm]
\dehat{q}_k(t)+\mathcal{P}(\dehat{\tht})[\dehat{q}_k](t)
=\dehat{w}(t)+\mathcal{P}(\dehat{\tht})[\dehat{q}_k](t)-\mathcal{P}_k(\dehat{\tht})[\dehat{q}_k](t).
\end{split}
\end{equation}
By virtue of Lemma \ref{p34}, we conclude that the limits $q = \lim_{k\to \infty} q_k$, $\dehat{q} = \lim_{k\to \infty} \dehat{q}_k$
are solutions to \eqref{eq77}. Furthermore, by Corollary \ref{c2} and \eqref{muj} we have for all $k\in \N$ that
\begin{align} \nonumber
|q_k(t)-\dehat{q}_k(t)| &\le \rho_k\,\left(\mot{w-\dehat{w}} + \left(\sum_{j=1}^k K_j\right)\Mot{\tht - \dehat{\tht}}]\right)\\ \label{ce2a}
&\le \expe^M\,\big(\mot{w-\dehat{w}} + M_1\,\Mot{\tht-\dehat{\tht}}\big),
\end{align}
and it suffices to pass to the limit as $k \to \infty$ to complete the proof.
\end{proof}

%%%%%%%%%%%%%%%%%%%%%%%%%%%%%%%%%%%%%%%%%%%%%%%%%%%%%%%%%

\section{Higher regularity}\label{regu}

For applications mentioned in Section \ref{phys}, we need to know that higher regularity of the inputs $w$, $\tht$
implies higher regularity of the solution $q$ to Eq.~\eqref{eq}. We denote by $C[0,T]$ the space
of continuous functions on $[0,T]$ and by $C([0,T];\R^L)$ the space of continuous vector functions.
The result can be stated as follows.

\begin{proposition}\label{pcont}
Assume that $\psi$ satisfies Hypothesis \ref{basic} and let $\tht\in C([0,T];\R^L)$ and $w\in C[0,T]$ be given.
Then the solution $q$ to \eqref{eq} belongs to $C[0,T]$. If, in addition, $\tht$ and $w$ are absolutely continuous,
then $q$ is absolutely continuous.
\end{proposition}

\begin{proof}
By Theorem \ref{p37}, a unique solution $q\in G_R[0,T]$ to \eqref{eq} exists.
To see that $q\in C[0,T]$ it remains to show that $q$ is left-continuous. Given $t_0\in(0,T]$ and $\varepsilon>0$,
we can choose $\delta>0$ such that for every $t_0 \in [\delta, T]$ we have 
\begin{equation}\label{w}
\expe^M\big(|w(\cdot)-w(t_0-\delta)|_{[t_0-\delta, t_0]} + M_1\|\tht(\cdot)-\tht(t_0-\delta)\|_{[t_0-\delta, t_0]}\big) < \varepsilon.
\end{equation}
For an arbitrary $0<h<\delta$ consider $\dehat{w}:[0,t_0]\to\R$ and $\dehat{\tht}:[0,t_0]\to\R^L$ defined by
\[
\dehat{w}(t)=
\begin{cases}
w(t)&t\in [0,t_0-h],
\\
w(t_0-h)&t\in[t_0-h,t_0],
\end{cases}
\qquad
\dehat{\tht}(t)=
\begin{cases}
\tht(t)&t\in [0,t_0-h],
\\
\tht(t_0-h)&t\in[t_0-h,t_0],
\end{cases}
\]
With $\dehat{w}$ and $\dehat{\tht}$, we associate the solution $\dehat{q}$ of the second equation of \eqref{eq77}.
By \eqref{eq78}, we have $\dehat{q}(t) = q(t)$ for $t \in [0,t_0-h]$. By definition \eqref{P} and \eqref{play},
the constant extension $q^*(t)=q(t_0-h)$ for $t\in [t_0-h,t_0]$ is a solution of the second equation of \eqref{eq77}.
Since the solution is unique, we have
\[
\dehat{q}(t)=
\begin{cases}
q(t)&t\in [0,t_0-h],
\\
q(t_0-h)&t\in[t_0-h,t_0]
\end{cases}
\]
and from \eqref{eq78} we infer that
\be{co1}
|q(t_0)- q(t_0-h)| \le \expe^M\big(|w(\cdot)-w(t_0-h)|_{[t_0-h, t_0]}
+ M_1\|\tht(\cdot)-\tht(t_0-h)\|_{[t_0-h, t_0]}\big) < \varepsilon.
\ee
which proves the continuity of $q$.

To prove the absolute continuity, we consider a finite sequence $0 \le a_1 < b_1 \le \dots \le a_n < b_n \le T$. By \eqref{co1},
we have on every interval $[a_i, b_i]$ the inequality
\be{co2}
|q(b_i)- q(a_i)| \le \expe^M\big(|w(\cdot)-w(a_i)|_{[a_i, b_i]} + M_1\|\tht(\cdot)-\tht(a_i)\|_{[a_i, b_i]}\big).
\ee
The absolute continuity of $q$ in the case when $w$ and $\tht$ are absolutely continuous now follows from the definition
of absolute continuity. In this case we also have
\be{co3}
|\dot q(t)| \le \expe^M\big(|\dot w(t)| + M_1\|\dot\tht(t)\|\big)\ \mbox{ a.\,e.},
\ee
where the dot denotes the derivative with respect to $t$.
\end{proof}

\end{document}